\documentclass{article}

\usepackage{amsmath}
\usepackage{amssymb}
\usepackage{mathtools}
\usepackage{amsthm}
\usepackage{graphicx}
\usepackage{amscd}
\usepackage{epic, eepic}
\usepackage{url}
\usepackage{color}
\usepackage[utf8]{inputenc} 
\usepackage{comment}
\usepackage{enumerate}   
\usepackage{todonotes}
\usepackage{graphicx}
\usepackage{epstopdf}
\usepackage{enumitem}
\usepackage{tikz}
\usepackage[top=25mm, bottom=25mm, left=28mm, right=28mm]{geometry}
\usepackage[bookmarks=false,hidelinks]{hyperref}

\makeatletter
\renewenvironment{proof}[1][\proofname] {\par\pushQED{\qed}\normalfont\topsep6\p@\@plus6\p@\relax\trivlist\item[\hskip\labelsep\bfseries#1\@addpunct{.}]\ignorespaces}{\popQED\endtrivlist\@endpefalse}
\makeatother

\newtheorem{theorem}{\bf Theorem}[section]
\newtheorem{lemma}[theorem]{\bf Lemma}
\newtheorem{claim}[theorem]{\bf Claim}
\newtheorem{corollary}[theorem]{\bf Corollary}
\newtheorem{proposition}[theorem]{\bf Proposition}

\newenvironment{custhm}[1]{\trivlist\item[\hskip\labelsep\textbf{Theorem~#1.}]\em}{\endtrivlist}

\theoremstyle{definition}

\newtheorem*{remark*}{\bf Remark}

\newcommand{\eps}{\varepsilon}

\DeclareMathOperator{\trace}{tr}
\DeclareMathOperator{\disc}{disc}

\DeclareMathOperator{\surp}{sp}

\title{Tight bounds for positive discrepancy via eigenvalues}
\author{Oliver Janzer\thanks{Institute of Mathematics, EPFL, Lausanne, Switzerland. Email: \textbf{oliver.janzer@epfl.ch}} \and Istv\'an Tomon\thanks{Department of Math. and Math. Stat., Ume\r{a} University, Ume\r{a}, Sweden. Email: \textbf{istvan.tomon@umu.se}.\\ Supported by the Swedish Research Council  grant VR 2023-03375.} \and Fredy Yip\thanks{Trinity College, University of Cambridge, United Kingdom. Email: \textbf{fy276@cam.ac.uk}.}}
\date{}

\begin{document}

\maketitle

\begin{abstract}
    Given an $n\times n$ symmetric matrix $M$ with largest eigenvalue $\lambda_1\geq 0$, it is easy to show that the solution of the optimisation problem $\max_{v\in [-1,1]^n}v^TMv$ is at most $\lambda_1 n$.  We prove the following converse: if every $n'\times n'$ principal submatrix of $M$  has maximal eigenvalue at least $\lambda$, then $\max_{v\in [-1,1]^n}v^TMv\geq \lambda(n-n'+1)$. This allows us to lift results about the second largest and the smallest eigenvalue of graphs to the (a priori more difficult) setting of positive discrepancy and MaxCut. Using this new connection, we prove that every $n$-vertex $m$-edge graph that is far from a disjoint union of cliques has a cut of size at least $m/2+n^{5/4-o(1)}$, which is sharp up to the $o(1)$-term. Moreover, we prove that every $d$-regular $n$-vertex graph has bisection width at most $dn/4-\Omega_{\eps}(d^{1/3}n)$ for $d\leq (1-\eps)n/2$, which is optimal for $d=\Omega(n)$. This confirms a conjecture of R\"aty, Sudakov and Tomon.
\end{abstract}

\section{Introduction}

Discrepancy measures the uniformity of mathematical objects, and its study has its roots in the classical theory of irregularities of distributions \cite{Beck1987}.
In the case of a graph, a fundamental question is how uniformly its edges can be distributed. Many central problems in combinatorics boil down to studying this question from various perspectives, such as pseudorandomness, expansion and regularity.

In 1988, Erd\H os, Goldberg, Pach, and Spencer \cite{erdos1988cutting} introduced the following notion of discrepancy, which is the starting point of our interests. Given a graph $G$ with $n$ vertices and edge density $p = e(G)/ \binom{n}{2}$, the \emph{discrepancy} of $G$ is defined as $$\disc(G) = \max_{U\subseteq V(G)} \left| e(G[U]) -p \binom{|U|}{2}\right|.$$
In other words, $\disc(G)$ measures how much the edge counts in the induced subgraphs of $G$ deviate from their expected number. Already in 1972, Erd\H os and Spencer~\cite{erdos1972imbalances} considered the $p = 1/2$ subcase, and proved that the discrepancy is always at least $\Omega(n^{3/2})$, which is tight for a random graph of edge density $1/2$. Later, Erd\H os, Goldberg, Pach, and Spencer \cite{erdos1988cutting} extended this by showing that $\disc(G) = \Omega(\sqrt{pn^3})$ as long as $1/n\leq p \leq 1/2$. This is again tight by virtue of the corresponding random graph. As the discrepancy of a graph is the same as the discrepancy of its complement, this implies a lower bound for all values of $p\in [1/n,1-1/n]$, which is also tight.

This result guarantees that in every graph there is an induced subgraph with either considerably more edges than expected, or one with considerably fewer edges. It is natural to wonder to what extent do both outcomes hold separately. The \emph{positive discrepancy} of a graph is defined as
$$\disc^+(G) = \max_{U\subseteq V(G)} e(G[U]) -p \binom{|U|}{2},$$
and analogously, the \emph{negative discrepancy} is defined as
$$\disc^-(G) = \max_{U\subseteq V(G)} p \binom{|U|}{2} - e(G[U]).$$
Observe that $\disc(G)=\max(\disc^+(G),\disc^-(G))$. Positive and negative discrepancy have been studied for decades, and they are  closely related to other important graph parameters such as MaxCut and bisection width.

\subsection{Positive discrepancy and bisection width}

If $G$ is the Tur\'an graph $T_r(n)$, that is, the complete balanced $r$-partite graph on $n$ vertices, then the positive discrepancy of $G$ is only $\Theta(n)$. This is much smaller than the $\Omega(n^{3/2}/r^{1/2})$ bound guaranteed for the discrepancy of graphs of the same density. As Tur\'an graphs have edge density at least $1/2$, it is natural to wonder whether graphs of edge density less than $1/2$ have large positive discrepancy. Verstra\"ete \cite[Conjecture A]{Verstraete} conjectured that if $1/n<p<1/2$, then the positive discrepancy achieves the same bound as the discrepancy, that is, $\disc^{+}(G)=\Omega(p^{1/2}n^{3/2})$. An important result in this direction was proved by Alon \cite{alon1997edge} in 1997: if $G$  is a  $d$-regular $n$-vertex graph with $d=O(n^{1/9})$, then the \emph{bisection width} of $G$ is at most $dn/4-\Omega(d^{1/2}n)$. Here, the bisection width of a graph is the minimal number of edges crossing a balanced bipartition of the graph. As observed in \cite{raty2026positive}, for $d$-regular graphs, this is equivalent to the statement that $\disc^+(G)=\Omega(n\sqrt{d})$. Thus, the result of Alon confirms the conjecture of Verstra\"ete for $p\ll n^{-8/9}$ in the case of regular graphs, by noting that $p=d/(n-1)$. A closely related result of Bollob\'as and Scott \cite{bollobas2006discrepancy} proved that for any $n$-vertex graph $G$ with average degree $1\leq d\leq n/2$, we have $\disc^+(G)\disc^-(G)=\Omega (n^2d)$. Observe that this, in particular, implies the result of Erd\H os, Goldberg, Pach and Spencer stating that $\disc(G)=\Omega(n\sqrt{d})$, and  it also implies that $\disc^+(G)=\Omega(n)$.

A few years ago, R\"aty, Sudakov and Tomon \cite{raty2026positive} showed that the  minimum of $\disc^+(G)$ among $n$-vertex graphs of average degree $d$ exhibits a strange dependence on $d$ and $n$, in particular observing that the conjecture of Verstra\"ete \cite{Verstraete} is not true. However, they showed that for $p<1/2$, one can indeed achieve substantial improvement over the general bound $\Omega(n)$. Their result is summarised in the following theorem.

\begin{theorem}[R\"aty--Sudakov--Tomon \cite{raty2026positive}] \label{thm:three regimes}
    Let $G$ be an $n$-vertex graph with average degree $d$ and let $\eps>0$. Then
    $$
\disc^+(G)=
\begin{cases}
\Omega(nd^{1/2}) & \text{ if } 1\leq d\leq n^{2/3}\\
\Omega(n^2/d) & \text{ if } n^{2/3}\leq d\leq n^{4/5}\\
\Omega_{\eps}(nd^{1/4}/\log n) & \text{ if } n^{4/5}\leq d\leq (\frac{1}{2}-\eps)n.
\end{cases}
$$
These bounds are tight for $1\leq d\leq n^{3/4}$.
\end{theorem}

The tightness of their result remained open in the range $n^{3/4}\leq d\leq  (1/2-\eps)n $. From above, the best known construction of Davis, Huczynska, Johnson, and Polhill \cite{davis2024denniston} provides graphs of constant edge density strictly less than $1/2$ with positive discrepancy $O(n^{4/3})$. Motivated by this construction and their lower bounds for the second largest eigenvalue of regular graphs (to be discussed in more detail in the next subsection), R\"aty, Sudakov and Tomon \cite{raty2026positive} conjectured that $\disc^+(G)=\Omega_{\eps}(nd^{1/3})$. We prove this conjecture.

\begin{theorem} \label{thm:positive disc}
    Let $\eps>0$ and let $G$ be an $n$-vertex graph of average degree $d$, where $1\leq d\leq (1/2-\eps)n$. Then $$\disc^+(G)=\Omega_{\eps}(nd^{1/3}).$$
    Moreover, if $G$ is $d$-regular, then the bisection width of $G$ is at most 
    $$\frac{dn}{4}-\Omega_{\eps}(nd^{1/3}).$$
    
\end{theorem}

A structural result about graphs of small discrepancy was also proposed by R\"aty, Sudakov and Tomon~\cite{raty2026positive}. They conjectured that being close to a Tur\'an graph is the only obstruction for large positive discrepancy. We say that $G$ is \emph{$\delta$-close} to a family of graphs $\mathcal{H}$ if the edit distance to some member of $\mathcal{H}$ is at most $\delta n^2$. Otherwise, we say that $G$ is \emph{$\delta$-far} from $\mathcal{H}$.  With this notation in our hand, \cite{raty2026positive} conjectured that for every $\delta>0$, if $G$ is $\delta$-far from every Tur\'an graph (including the empty graph), then $\disc^+(G) \geq \Omega_{\delta}(n^{5/4})$. Here, the exponent $5/4$ cannot be improved, due to a construction of de~Caen~\cite{decaen}. Recently, Jin, Milojevi\'c, Tomon and Zhang \cite{jin2025small} proved the weaker version of this conjecture that $\disc^+(G) \geq \Omega_{\delta}(n^{1+c})$ for some absolute constant $c>0$.

We point out that in the course of proving this theorem and its variants, they have developed several powerful methods that allowed them, among other results, to prove a polynomial bound for the celebrated Chowla cosine problem \cite{Chowla1965} (see \cite{bedert2025polynomial} for independent work on the same problem). A slightly more detailed discussion of their results will be given in the next subsection. We are now ready to state our second main result, which is an asymptotic resolution of the conjecture of R\"aty, Sudakov and Tomon \cite{raty2026positive}.

\begin{theorem} \label{thm:disc close to Turan}
    For every $\delta > 0$, if $n$ is sufficiently large in terms of $\delta$ and $G$ is an $n$-vertex graph which is $\delta$-far from every Tur\'an graph, then $\disc^+(G) \geq n^{5/4-o(1)}$.
\end{theorem}

\subsection{Maximum cut}
 
A cut in a graph is a bipartition of its vertex set, and the size of a cut is the number of edges going across the partition. Let us write $\textrm{mc}(G)$ for the size of the largest cut in $G$. A simple probabilistic argument shows that $\textrm{mc}(G)\geq e(G)/2$ holds for every graph $G$, so a more convenient parameter to study is the \emph{surplus} of the graph, defined as $\surp(G)=\textrm{mc}(G)-e(G)/2$. MaxCut and surplus are central notions in both Combinatorics and Theoretical Computer Science; see \cite{edwards1973some,erdHos1975problems,goemans1995improved}.

As observed by R\"aty, Sudakov and Tomon \cite{raty2026positive}, for every regular graph $G$, $\disc^+(G)$ and $\surp(\overline{G})$ are within a constant factor of each other, where $\overline{G}$ denotes the complement of $G$. Moreover, even without the regularity assumption, $\surp(\overline{G})=O(\disc^+(G))$ (see Lemma~\ref{lem:sp to disc} for the precise statement). Given this connection, it is not surprising that we can establish a result similar to Theorem~\ref{thm:disc close to Turan} for surplus as well, addressing another conjecture of \cite{raty2026positive}, and improving a result of \cite{jin2025small}.

\begin{theorem} \label{thm:sp close to Turan}
    For every $\delta > 0$, if $n$ is sufficiently large in terms of $\delta$ and $G$ is an $n$-vertex graph which is $\delta$-far from every disjoint union of cliques, then $\surp(G) \geq n^{5/4-o(1)}$.
\end{theorem}

The exponent $5/4$ here is tight just as in Theorem \ref{thm:disc close to Turan}. Note that the cliques need not have equal sizes, as any disjoint union of cliques has small surplus. In contrast, a complete multipartite graph has small positive discrepancy only if the parts are roughly of the same size (see Lemma \ref{lem:from complete multipartite to Turan} below).

These results are strongly motivated by the problem of studying the surplus of $H$-free graphs, which was initiated by Erd\H os and Lov\'asz \cite{erdHos1975problems} in the 70's. This problem has received significant attention since then, see, e.g. \cite{alon1996bipartite,alon2003maximum,alon2005maxcut,balla2024maxcut,carlson2021lower,glock2023new,jin2025small,shearer1992note}. Here we focus on the case where $H$ is a clique. A celebrated result of Alon \cite{alon1996bipartite} states that every $K_3$-free graph with $m$ edges has surplus $\Omega(m^{4/5})$, and this is best possible. In general, Alon, Bollob\'as, Krivelevich and Sudakov \cite{alon2003maximum} showed that for each $t$ there exists some $\eps_t>0$ for which every $K_t$-free graph with $m$ edges has surplus $\Omega_t(m^{1/2+\eps_t})$. They made the conjecture that in fact any such graph has surplus $\Omega_t(m^{3/4+\eps'_t})$, for some $\eps'_t>0$, which would be optimal by random graphs with appropriate edge densities. However, until very recently, it was not even known that such graphs have surplus $\Omega_t(m^{1/2+c})$ for some absolute constant $c>0$. It was proved by Jin, Milojevi\'c, Tomon and Zhang \cite{jin2025small} that any $K_t$-free graph with $m$ edges has surplus $\Omega_t(m^{0.51})$, breaking the $1/2$ barrier.

In fact, they prove a variety of results in this direction, all saying that if a graph has no large clique, then its surplus or the absolute value of its least eigenvalue has to be large. The proofs in the eigenvalue setting are simpler (though by no means simple) and provide better bounds. With our new machinery, Theorem \ref{thm:eigenvalue to surplus}, the deduction of bounds for surplus from their bounds for the smallest eigenvalue is almost immediate. We demonstrate this by proving the following improved bound for the surplus of $K_t$-free graphs.

\begin{theorem} \label{thm:Kt free surplus bound}
    For any fixed positive integer $t$, if $G$ is a $K_t$-free graph with $m$ edges, then $\surp(G)\geq m^{0.55 - o(1)}$ as $m\rightarrow \infty$.
\end{theorem}

We remark that we did not optimise the constant $0.55$, and indeed a better bound can be obtained by using different inputs from \cite{jin2025small} and a slightly more involved deduction. In fact, in companion work of Yip \cite{yip2026lower}, the bound $\surp(G)\geq m^{0.614 - o(1)}$ is proved using our machinery as well as a new result on the eigenvalue side.

\subsection{Eigenvalues}

Let $G$ have eigenvalues $\lambda_1\geq \dots \geq \lambda_n$. It is well known \cite{alon1996bipartite,delorme1993laplacian,mohar1990eigenvalues} that $$\surp(G)\leq \frac{1}{4}n|\lambda_n|.$$ Analogously, as observed by R\"aty, Sudakov and Tomon \cite{raty2026positive}, for every $d$-regular graph, we have 
$$\disc^+(G)\leq \frac{1}{2}n\lambda_2+d.$$ These results show that lower bounds on the surplus and positive discrepancy yield lower bounds for $|\lambda_n|$ and $\lambda_2$. The weaker versions of Theorems \ref{thm:positive disc} and \ref{thm:sp close to Turan} in which $\surp(G)$ and $\disc^+(G)$ are replaced with the corresponding eigenvalue have already been established prior to our work, at least for regular graphs. More precisely, the following eigenvalue version of Theorem \ref{thm:positive disc} was proved by Balla \cite{balla2025equiangular}. The result was later reproved by R\"aty, Sudakov and Tomon \cite{raty2026positive}, and a particularly short proof was given in \cite{balla2023note}. It also follows fairly easily from a result of Ihringer \cite{ihringer2023approximately}.
\begin{theorem}[Balla \cite{balla2025equiangular}] \label{thm:eigenvalue dense}
    Let $\eps>0$ and let $G$ be a $d$-regular $n$-vertex graph where $1\leq d\leq (1/2-\eps)n$. Then $$\lambda_2(G)=\Omega_{\eps}(d^{1/3}).$$
\end{theorem}

The eigenvalue version of Theorem \ref{thm:disc close to Turan} is substantially more involved and was proved in a recent work of Jin, Milojevi\'c, Tomon and Zhang \cite{jin2025small}.

\begin{theorem}[Jin--Milojevi\'c--Tomon--Zhang {\cite[Theorem 1.6]{jin2025small}}] \label{thm:lambda_2 close to Turan}
    For every $\delta > 0$, if $n$ is sufficiently large in terms of $\delta$ and $G$ is a regular $n$-vertex graph which is $\delta$-far from every Tur\'an graph, then $\lambda_2(G) \geq n^{1/4-o(1)}$.
\end{theorem}

The main goal of this paper is to show that extremal results involving the second and smallest eigenvalue can be transferred to results about the positive discrepancy and MaxCut, respectively. Note that $\surp(G)$ can be substantially smaller than $n|\lambda_n|$, and similarly $\disc^+(G)$ can be much smaller than $n\lambda_2$. For example, if $G$ is the disjoint union of a clique of size $n/2-1$ and a star of size $n/2+1$, then $\surp(G)=O(n)$, while $\lambda_n=-\sqrt{n/2}$. The issue is that $G$ has a sparse induced subgraph, namely the star, which forces $|\lambda_n|$ to be very large, while it has no noticeable effect on the size of the MaxCut. Our main technical result shows that if in every large induced subgraph of $G$ the smallest eigenvalue has large absolute value, then we can conclude that the surplus is large.

\begin{theorem} \label{thm:eigenvalue to surplus}
    Let $G$ be a graph on $n$ vertices, and let $n'\leq n$ be a positive integer. Assume that for any induced subgraph $H$ of $G$ on at least $n'$ vertices, the absolute value of the least eigenvalue of $H$ is at least~$\lambda$. Then $\surp(G)\geq (n - n')\lambda/4$. 
\end{theorem}

Theorem \ref{thm:eigenvalue to surplus}, when combined with the smallest eigenvalue version of Theorem \ref{thm:lambda_2 close to Turan}, yields a very short proof of Theorems \ref{thm:disc close to Turan} and \ref{thm:sp close to Turan}. Combining Theorem~\ref{thm:eigenvalue to surplus} with a stronger eigenvalue estimate in place of Theorem~\ref{thm:lambda_2 close to Turan}, a companion work of Yip~\cite{yip2026lower} proves Theorems~\ref{thm:disc close to Turan} and~\ref{thm:sp close to Turan} without the $o(1)$ term in the exponent, fully resolving the conjecture of \cite{raty2026positive}. We remark that a different proof of Theorem \ref{thm:disc close to Turan} without the $o(1)$ term was obtained independently by Xu and Zhou \cite{xu2026positive}.

Our proof of Theorem \ref{thm:positive disc} relies on a slightly more general version of Theorem \ref{thm:eigenvalue to surplus}, which we now introduce.

\subsection{Matrix optimisation}

A fundamental result in operator theory is Grothendieck's inequality, which states that there exists a universal constant $1.67<K_G< 1.79$ with the following property. Let $M$ be an $m\times n$ matrix. A common optimisation problem is to calculate
$$||M||_{\infty\rightarrow 1}=\left|\max_{u\in  [-1,1]^m, v\in [-1,1]^n} u^T Mv\right|=\left|\max_{u_1,\dots,u_m,v_1,\dots,v_n\in [-1,1]}\sum_{i=1}^m\sum_{j=1}^n M_{i,j}u_iv_j\right|.$$
One can relax the right-hand side by replacing $u_i,v_j$ with unit vectors, and consider the semidefinite optimisation problem
$$||M||_{G}=\left|\max_{\mathbf{u}_1,\dots,\mathbf{u}_m,\mathbf{v}_1,\dots,\mathbf{v}_n\in \mathbb{S}^N}\sum_{i=1}^m\sum_{j=1}^n M_{i,j}\langle \mathbf{u}_i,\mathbf{v}_j\rangle \right|,$$
where $\mathbb{S}^N$ denotes the $N$-dimensional unit sphere, and $N$ is arbitrary. Then, Grothendieck's inequality states that this relaxation is a surprisingly good approximation of the original problem, that is,
$$||M||_{\infty\rightarrow 1}\leq ||M||_{G}\leq K_G ||M||_{\infty\rightarrow 1}.$$

Inspired by graph optimisation problems, Charikar and Wirth \cite{charikar2004maximizing} introduced a symmetric variant of Grothendieck's inequality. Given an $n\times n$ symmetric real matrix $M$, define the \emph{positive discrepancy} of $M$ as
\begin{equation*}
    \disc^+(M) := \max_{v\in [-1, 1]^n}v^TMv. 
\end{equation*}
Similarly as before, one might compare this optimisation problem to the semidefinite relaxation
\begin{equation}\label{equ:graph_groth}
    \disc_G^+(M)=\max_{\|\mathbf{v}_i\|\leq 1, i\in [n]} \sum_{i=1}^n\sum_{j=1}^n M_{i,j}\langle v_i,v_j\rangle=\max_{X\succeq 0, X_{i,i}\leq 1, i\in [n]} \langle M,X\rangle,
\end{equation}
where the maximum on the right-hand-side is taken over all $n\times n$ positive semidefinite matrices $X$ with diagonal entries at most 1. In \cite{charikar2004maximizing}, it is proved that 
$$\disc^{+}(M)\leq \disc_G^+(M)\leq O(\disc^+(M)\log n),$$
and Alon, Makarychev, Makarychev, and Naor \cite{alon2006quadratic} showed that the extra logarithmic factor is unavoidable in general. This inequality played a central role in the recent results \cite{jin2025small,raty2026positive} about the MaxCut, minimum bisection and positive discrepancy of graphs. More precisely, for certain suitable matrices $M$ associated with a graph $G$, it was proved that $\disc^+(M)$ is large  by constructing test matrices $X$ for (\ref{equ:graph_groth}) based on the spectral decomposition of $M$.

The novelty in our paper is taking a different, more elementary route to show that $\disc^{+}(M)$ is large, and we do not invoke Grothendieck's inequality and related results in any capacity. Note that if $\lambda_1\geq 0$ is the largest eigenvalue of $M$, then one has $\disc^{+}(M)\leq \lambda_1 n$. We prove the converse that if every not too small principal submatrix of $M$ has large maximum eigenvalue, then that implies that $\disc^{+}(M)$ is also large.

\begin{theorem} \label{thm:matrix bridge}
    Let $M$ be an $n\times n$ symmetric real matrix, and let $n'\leq n$ be a positive integer. Assume that the largest eigenvalue of every  $n'\times n'$ principal submatrix of $M$ is at least $\lambda$. Then $$\disc^+(M)\geq \lambda (n - n'+1).$$ 
\end{theorem}

\sloppy Related results have been proven before. Tropp \cite[Proposition 5.2]{tropp2009column} studies the quantity $\max_{v\in [-1,1]^n}|v^TMv|=\max(\disc^+(M),\disc^+(-M))$, and proves a variant of Theorem \ref{thm:matrix bridge} by reducing it to Grothendieck's inequality. We remark that this quantity is also equal to $\|M\|_{\infty \rightarrow 1}$ up to a factor of two \cite[Corollary 2.6]{friedland2020symmetric}. The main difference between our result and Tropp's is that we lower bound the one-sided discrepancy, which is important for our applications.

In the asymmetric case, Le, Levina and Vershynin \cite[Theorem 2.1]{le2015sparse} proved that if $M$ is an ${m\times n}$ matrix and the largest singular value of every $m'\times n'$ submatrix is at least $\sigma$, then $||M||_{\infty\rightarrow 1}\geq \frac{\sigma}{K_G} \sqrt{(m-m')(n-n')}$. (To be precise, their theorem is stated with $\frac{1}{2}$ instead of $\frac{1}{K_G}$, however, this slightly stronger result follows from their proof.) We slightly improve this result by a completely elementary argument. The next theorem will not be used later, but it might be of independent interest.

\begin{theorem}\label{thm:groth}
     Let $M$ be an $m\times n$ real matrix, and let $m'\leq m$ and $n'\leq n$ be positive integers. Assume that the largest singular value of any $m'\times n'$ submatrix of $M$ is at least $\sigma$. Then $$||M||_{\infty\rightarrow 1}\geq \sigma\sqrt{(m-m'+1)(n - n'+1)}.$$ 
\end{theorem}

\medskip

\noindent
\textbf{Organisation.} The rest of this paper is organised as follows. In the next section we collect a few preliminary lemmas we will use in our proofs. In Section \ref{sec:eigenvalue to surplus}, we prove Theorems \ref{thm:eigenvalue to surplus}, \ref{thm:matrix bridge} and \ref{thm:groth}. In Section \ref{sec:close to Turan}, we deduce Theorems \ref{thm:disc close to Turan} and \ref{thm:sp close to Turan}. In Section \ref{sec:positive discrepancy}, we prove Theorem \ref{thm:positive disc}. Finally, in Section \ref{sec:ktfree}, we prove Theorem \ref{thm:Kt free surplus bound}.

\section{Preliminaries} \label{sec:preliminaries}

In this section we collect a few standard lemmas that we use in our proofs. We start with a simple lemma relating positive discrepancy to the surplus of the complement graph.

\begin{lemma} \label{lem:sp to disc}
    Let $G$ be a graph on $n$ vertices. Then 
    $$\surp\left(\overline{G}\right)\leq 2\disc^+(G) + \frac{n}{4}.$$ 
\end{lemma}

\begin{proof}
    Let $T\subseteq V(G)$ such that the cut $(T,T^c)$ is maximal in $\overline{G}$, that is, 
    $$ \frac{1}{2}e\left(\overline{G}\right) + \surp\left(\overline{G}\right)=e_{\overline{G}}(T, T^c).$$
    Equivalently,
    $$ \frac{1}{2}e\left(\overline{G}\right) - \surp\left(\overline{G}\right)=e_{\overline{G}}(T) + e_{\overline{G}}(T^c).$$ Let $p$ denote the density of $G$, then $$e_G(T)\leq p\binom{|T|}{2} + \disc^+(G)\quad\text{and}\quad e_G(T^c)\leq p\binom{|T^c|}{2} + \disc^+(G).$$
    By combining the previous estimates, we get
\begin{align*}
\frac{1}{2}e(\overline{G})-\surp(\overline{G})&=e_{\overline{G}}(T) + e_{\overline{G}}(T^c)\\
&=\binom{|T|}{2}-e_{G}(T) +\binom{|T^c|}{2}- e_{G}(T^c)\\
&\geq  (1-p)\left(\binom{|T|}{2}+\binom{|T^c|}{2}\right)-2\disc^+(G).
\end{align*} 
    Writing $e(\overline{G})=(1-p)\binom{n}{2}$ and rearranging this inequality, we obtain
    \begin{equation*}
        (1 - p)\left(\binom{|T|}{2} + \binom{|T^c|}{2} - \frac{1}{2}\binom{n}{2}\right)\leq 2\disc^+(G) -\surp\left(\overline{G}\right).\\
    \end{equation*}
    Since 
    $$\binom{|T|}{2} + \binom{|T^c|}{2} - \frac{1}{2}\binom{n}{2}\geq -\frac{n}{4}$$ by the Cauchy-Schwarz inequality, the lemma follows. 
\end{proof}

Next, we show that if graphs $G$ and $H$ are close in edit distance, then their positive discrepancies are similar.

\begin{lemma} \label{lem:pdisc change}
    Let $G$ and $H$ be graphs of edit distance $s$. Then $$\disc^+(H)\leq \disc^+(G)+s.$$
\end{lemma}

\begin{proof}
    It is sufficient to prove the lemma for $s=1$; the general case follows by iteration. Assume that $G$ and $H$ have the same vertex set $V$. Choose $U\subseteq V$ such that $$\disc^+(H)=e_H(U)-\binom{|U|}{2}\frac{e(H)}{\binom{n}{2}}.$$ If $H$ contains one more edge than $G$, then $e_H(U)\leq e_G(U)+1$ and $e(H)=e(G)+1$. If $H$ contains one fewer edge than $G$, then $e_H(U)\leq e_G(U)$ and $e(H)=e(G)-1$. In both cases,
    $$\disc^+(H)=e_H(U)-\frac{\binom{|U|}{2}}{\binom{n}{2}}e(H)\leq e_G(U)-\frac{\binom{|U|}{2}}{\binom{n}{2}}e(G)+1\leq \disc^+(G)+1,$$
    as desired.
\end{proof}

We also make use of two simple results from \cite{raty2026positive}.

\begin{lemma}[{\cite[Claim 4.1]{raty2026positive}}] \label{lem:pdisc from vector}
    For a graph $G$ with edge density $p$ and adjacency matrix $A$, we have
    $$\disc^+(G)\geq \frac{1}{8}\disc^+(A - p(J - I)).$$
\end{lemma}

\begin{lemma}[{\cite[Lemma 2.4]{raty2026positive}}] \label{lem:if weight on high degree}
    Let $\delta \in (0,1)$. Then there exist $c_1,c_2 > 0$ such that the following holds. Let $G$ be an $n$-vertex graph with average degree $d$. Let $X \subseteq V(G)$ such that the degree of every vertex in $X$ is at least $(1 + \delta)d$. Then $$\disc^+(G) \geq c_1(e(X) +e(X,X^c)) - c_2n.$$
\end{lemma}

Finally, we recall the following simple regularisation lemma from \cite{glock2023new}.

\begin{lemma}[{\cite[Lemma 2.5]{glock2023new}}] \label{lem:regularization}
    Let $\alpha$ and $\beta$ be real numbers such that $\beta > 0$, $\alpha < \beta$ and $\alpha+\beta \leq 2$. Then there exist positive constants $c$ and $C$ (depending on $\alpha,\beta$), such that the following holds. Let $G$ be a graph on $n$ vertices with average degree $d$. Then either $G$ has surplus at least $cn^{\alpha} d^{\beta}$ or $G$ has an induced subgraph $\tilde{G}$ with $\tilde{n}$ vertices and average degree $\tilde{d}$ where $\tilde{n}^{\alpha} \tilde{d}^{\beta} \geq cn^{\alpha}d^{\beta}$ and $\Delta(\tilde{G}) \leq C\tilde{d}$.
\end{lemma}

\section{From eigenvalues to surplus and positive discrepancy} \label{sec:eigenvalue to surplus}

In this section we prove Theorems \ref{thm:eigenvalue to surplus}, \ref{thm:matrix bridge} and \ref{thm:groth}.

\begin{proof}[Proof of Theorem \ref{thm:matrix bridge}]
     If $\lambda<0$, the statement is trivial, so we may assume $\lambda\geq 0$. Call a vector $v\in [-1, 1]^n\subseteq \mathbb{R}^n$ \emph{admissible} if $v^TMv\geq \lambda \|v\|_2^2$. Such vectors exist. For instance, the zero vector is admissible. Take $v$ to be an admissible vector maximising the size of the set $T = \{i\in [n]:|v_i| = 1\}$. If the size of this set is larger than $n - n'$, then 
    \begin{equation*}
        v^TMv\geq \lambda \|v\|_2^2 \geq \lambda (n - n'+1), 
    \end{equation*}
    which implies $\disc^+(M)\geq \lambda (n - n'+1)$, as needed. Therefore, we may assume that $|T| \leq  n - n'$. Let $S = T^c$. Since $|S| \geq  n'$, the maximum eigenvalue of $M[S\times S]$ is at least $\lambda$. In other words, there exists a non-zero vector $w\in \mathbb{R}^n$ supported on $S$ such that 
    \begin{equation*}
        w^TMw \geq \lambda\|w\|_2^2. 
    \end{equation*}
    For $t\in \mathbb{R}$, let $v_t = v + tw$. Note that
    \begin{equation*}
        f(t):= v_t^TMv_t - \lambda v_t^Tv_t = (w^TMw - \lambda w^Tw)t^2 + (2w^TMv - 2\lambda w^Tv)t + (v^TMv - \lambda v^Tv)
    \end{equation*}
    is convex. Let the interval $I = [a, b]$ be the range of values of $t$ for which $v_t\in [-1, 1]^n$. Note that 
    $$|\{i\in [n]:|(v_a)_i| = 1\}|\quad\text{and}\quad|\{i\in [n]:|(v_b)_i| = 1\}|$$ are both strictly greater than $|T|$. Therefore, by the assumption of the maximality of $v$, we may infer that neither $v_a$ nor $v_b$ is admissible. That is, $f(a)$ and $f(b)$ are both negative. This is a contradiction to the fact that $f(0)\geq 0$ and that $f$ is convex.
\end{proof}

We now deduce Theorem~\ref{thm:eigenvalue to surplus} from Theorem~\ref{thm:matrix bridge}. Note that when the diagonal entries of $M$ are non-negative, the function $v\mapsto v^TMv$ is entry-wise convex. Therefore, we may take the maximum of $v^TMv$ over $v\in \{\pm 1\}^n$ instead of $[-1, 1]^n$ in the definition of $\disc^+(M)$. In other words, we have
\begin{equation*}
    \disc^+(M) = \max_{v\in [-1, 1]^n}v^TMv = \max_{v\in \{\pm 1\}^n}v^TMv. 
\end{equation*}

\begin{proof}[Proof of Theorem \ref{thm:eigenvalue to surplus}]
     We apply Theorem~\ref{thm:matrix bridge} to $M = -A$, for which the conditions of Theorem~\ref{thm:matrix bridge} are satisfied. Therefore, there exists $v\in \{\pm 1\}^n$ for which
     \begin{equation*}
         -v^TAv = v^TMv\geq (n - n')\lambda. 
     \end{equation*}
     The desired result now follows since $-v^TAv/4$ is the surplus of the cut $\{i : v_i = 1\}, \{i : v_i = -1\}$. 
\end{proof}

Finally, we prove Theorem \ref{thm:groth}, whose proof is very similar to the proof of Theorem \ref{thm:matrix bridge}. 

\begin{proof}[Proof of Theorem \ref{thm:groth}]
Let $\theta=\sqrt{(n-n'+1)/(m-m'+1)}$ and for $(u,v)\in \mathbb{R}^m\times  \mathbb{R}^n$, define
$$F(u,v)=u^TMv-\frac{\sigma}{2}(\theta ||u||^2+\theta^{-1}||v||^2).$$
Choose a pair $(u,v)$ that maximises $F(u,v)$ on the compact set $[-1,1]^m\times [-1,1]^n$. Moreover, writing $S=\{i\in [m]: |u_i|=1\}$ and $T=\{j\in [n]: |v_j|=1\}$, among all maximisers, choose one that  maximises $|S|+|T|$. First, we prove that either $|S|\geq m-m'+1$ or $|T|\geq n-n'+1$. Assume to the contrary that $S^c=[m]\setminus S$ has size at least $m'$ and $T^c=[n]\setminus T$ has size at least $n'$. Then by our assumption, there exist unit vectors $(w,z)\in \mathbb{R}^m\times \mathbb{R}^n$ such that $w$ is supported only on $S^c$, $z$ is supported only on $T^c$, and $w^T M z\geq \sigma$. 

For $t\in \mathbb{R}$, consider the perturbations 
$$u_t=u+t\theta^{-1/2}w\quad\text{and}\quad v_t=v+t\theta^{1/2}z.$$
Let $I=[a,b]$ be the range of values of $t$ for which $u_t\in [-1,1]^m$ and $v_t\in [-1,1]^n$. Note that $0$ is an interior point of $I$. Moreover, 
$$|\{i\in [m]: |(u_a)_i|=1\}|+|\{j\in [n]:|(v_a)_j|=1\}|\quad\text{and}\quad|\{i\in [m]: |(u_b)_i|=1\}|+|\{j\in [n]:|(v_b)_j|=1\}|$$
are strictly larger than $|S|+|T|$. Consider the function
\begin{align*}
f(t)&=F(u_t,v_t)-F(u,v) \\
&=Lt+t^2(w^TMz-\sigma),
\end{align*}
where $L$ does not depend on $t$.
Then $f(0)=0$ and $f$ is a convex function, so we must have $f(a)\geq 0$ or $f(b)\geq 0$, contradicting the maximality of $(u,v)$.

Thus, we proved that either $|S|\geq m-m'+1$ or $|T|\geq n-n'+1$. Assume that $|S|\geq m-m'+1$; the other case can be handled similarly. For $\eps\in (0,1)$, consider
$$F(u,v)-F((1-\eps)u,v)=\eps (u^TMv)-\frac{\sigma}{2}\theta (2\eps-\eps^2) ||u||^2.$$
Taking $\eps\rightarrow 0$, the maximality of $F(u,v)$ implies that 
$$u^TMv\geq \sigma \theta ||u||^2\geq \sigma \theta |S|\geq \sigma \theta (m-m'+1)\geq \sigma \sqrt{(m-m'+1)(n-n'+1)},$$
as desired.
\end{proof}

\section{Disjoint union of cliques and Tur\'an graphs} \label{sec:close to Turan}

In this section, we prove Theorems \ref{thm:sp close to Turan} and \ref{thm:disc close to Turan}.
We start with the former. The crucial ingredient, in addition to Theorem \ref{thm:eigenvalue to surplus}, is the following eigenvalue version of Theorem \ref{thm:sp close to Turan}, proved in \cite{jin2025small}.

\begin{theorem}[Jin--Milojevi\'c--Tomon--Zhang {\cite[Theorem 1.4]{jin2025small}}] \label{thm:lambda_n close to Turan}
    Let $\gamma \in (0,1/4)$, let $\delta > 0$, and let $n$ be sufficiently large with respect to $\gamma$ and $\delta$. If $G$ is an $n$-vertex graph with $|\lambda_n| \leq n^{\gamma}$, then $G$ is $\delta$-close to the vertex-disjoint union of cliques.
\end{theorem}

We are now ready to prove Theorem \ref{thm:sp close to Turan}, which we restate in the following equivalent form.

\begin{custhm}{\ref{thm:sp close to Turan}'} 
    Let $\gamma \in (0,1/4)$, let $\delta > 0$, and let $n$ be sufficiently large with respect to $\gamma$ and $\delta$. If $G$ is an $n$-vertex graph with $\surp(G) \leq n^{1+\gamma}$, then $G$ is $\delta$-close to a disjoint union of cliques.
\end{custhm}

\begin{proof}
    We prove that if $G$ is $\delta$-far from every graph that is a disjoint union of cliques, then $\surp(G)\geq n^{1+\gamma}$. Fix some $\gamma_0\in (\gamma,1/4)$. Observe that every induced subgraph of $G$ on at least $(1-\delta/2)n$ vertices is $(\delta/2)$-far from a disjoint union of cliques. Hence, by Theorem \ref{thm:lambda_n close to Turan}, if $n$ is sufficiently large in terms of $\delta $ and $\gamma_0$, then each such induced subgraph of $G$ has smallest eigenvalue at least $(n/2)^{\gamma_0}$ in absolute value. Thus, by Theorem \ref{thm:eigenvalue to surplus}, $\surp(G)\geq \frac{\delta}{2}n\cdot (n/2)^{\gamma_0}/4\geq n^{1+\gamma}$ for sufficiently large $n$.
\end{proof}

Next, we prove Theorem \ref{thm:disc close to Turan} in the following equivalent form.

\begin{custhm}{\ref{thm:disc close to Turan}'} 
    Let $\gamma \in (0,1/4)$, let $\delta > 0$, and let $n$ be sufficiently large with respect to $\gamma$ and $\delta$. If $G$ is an $n$-vertex graph with $\disc^+(G) \leq n^{1+\gamma}$, then $G$ is $\delta$-close to a Tur\'an graph.
\end{custhm}

We prove this result in two steps. First, we use Theorem \ref{thm:sp close to Turan}' to show that under the assumptions of the theorem, $G$ is close to a complete multipartite graph (with not necessarily balanced parts).

\begin{proposition} \label{prop:disc close to multipartite}
    Let $\gamma \in (0,1/4)$, let $\eps > 0$, and let $n$ be sufficiently large with respect to $\gamma$ and $\eps$. If $G$ is an $n$-vertex graph with $\disc^+(G) \leq n^{1+\gamma}$, then $G$ is $\eps$-close to a complete multipartite graph.
\end{proposition}

\begin{proof}
    Let $\overline{G}$ be the complement of the graph $G$, and let $\gamma'\in (\gamma,1/4)$. By Lemma~\ref{lem:sp to disc}, 
    $$\surp(\overline{G})\leq 2n^{1 + \gamma} + \frac{n}{4}\leq n^{1 + \gamma'}$$ when $n$ is sufficiently large. Hence, by Theorem \ref{thm:sp close to Turan}', $\overline{G}$ is $\eps$-close to the vertex-disjoint union of cliques, as long as $n$ is sufficiently large. The result then follows by taking complements.
\end{proof}

Next, we prove that if a complete multipartite graph has small positive discrepancy, then it is close to a Tur\'an graph.

\begin{lemma} \label{lem:from complete multipartite to Turan}
    For every $\delta>0$ there exists $\eps>0$ such that if $n$ is sufficiently large and $H$ is a complete multipartite graph on $n$ vertices with $\disc^+(H)\leq \eps n^2$, then $H$ is $\delta$-close to a Tur\'an graph.
\end{lemma}

\begin{proof}
    Let $p$ be the edge density of $H$. If $p\geq 1-\delta$, then $H$ is $\delta$-close to the complete graph and we are done. Thus, we assume that $p<1-\delta$. Let $S$ be the union of parts of size at most $K=\eps^{1/3}n$ in $H$. Assuming that $|S|\geq \delta n/10$, we have 
    \begin{align*}
        \disc^+(H)&\geq e(S)-p\binom{|S|}{2}>e(S)-(1-\delta)\binom{|S|}{2}\\
        &\geq \binom{|S|}{2}-\frac{|S|}{K}\binom{K}{2}-(1-\delta)\binom{|S|}{2}\geq \frac{\delta}{10} |S|^2\\
        &\geq \frac{\delta^3}{1000} n^2\geq \eps n^2,
    \end{align*} where the fourth and last inequalities hold assuming $\eps$  is sufficiently small with respect to $\delta$. Hence, we may assume that $|S|< \delta n/10$. Let $T=V(H)\setminus S$. Note that it is sufficient to prove that $H[T]$ is $(\delta/2)$-close to a Tur\'an graph.

    Next, we prove that $\disc^+(H[T])\leq \disc^+(H)+o(n^2)$. Choose $T'\subseteq T$ such that $$e(T')-\mathbb{E}[e(T_{\textrm{rand}})]=\disc^+(H[T]),$$
    where $T_{\textrm{rand}}$ is a random subset of $T$, chosen from the uniform distribution on the subsets of size $|T'|$. Let $\alpha=|T'|/|T|$. Let $S_{\textrm{rand}}$ be a random subset of $S$ of size $\alpha |S|$. Let $V_{\textrm{rand}}$ be a random subset of $V(H)$ of size $\alpha |V(H)|$. Note that (since $H$ is complete between $S$ and $T$), we have $$\mathbb{E}[e(T'\cup S_{\textrm{rand}})]-\mathbb{E}[e(T_{\textrm{rand}}\cup S_{\textrm{rand}})]=e(T')-\mathbb{E}[e(T_{\textrm{rand}})]=\disc^+(H[T]),$$ which implies that $$\mathbb{E}[e(T'\cup S_{\textrm{rand}})]-\mathbb{E}[e(V_{\textrm{rand}})]=\disc^+(H[T])-o(n^2).$$
    From this, $\disc^+(H[T])\leq \disc^+(H)+o(n^2)$ follows.

    But then, we also have $\disc^+(H[T])\leq \eps n^2+o(n^2)$. Note that $H[T]$ is a complete multipartite graph; let $r$ denote the number of parts. Now $H[T]$ has an $r$-partite subgraph with parts of size $\eps^{1/3}n$ each. This subgraph shows that $$\disc^+(H[T])\geq \frac{1}{2}\left(1-\frac{1}{r}-\rho-o(1)\right)(r\eps^{1/3}n)^2,$$ where $\rho$ is the edge density of $H[T]$. Comparing the lower and upper bound for $\disc^+(H[T])$, we have $\rho\geq 1-1/r-\eps^{1/4}$ for large $n$. For sufficiently small $\eps$, this implies that $H[T]$ is $(\delta/2)$-close to an $r$-partite Tur\'an graph.
\end{proof}

We are now ready to prove Theorem \ref{thm:disc close to Turan}'.

\begin{proof}[Proof of Theorem \ref{thm:disc close to Turan}']
    Let $\eps>0$ (with $\eps<\delta$)  be a constant guaranteed by Lemma \ref{lem:from complete multipartite to Turan} applied with $\delta/2$ in place of $\delta$. By Proposition \ref{prop:disc close to multipartite}, if $n$ is sufficiently large in terms of $\gamma$ and $\eps$ and $G$ is an $n$-vertex graph with $\disc^+(G)\leq n^{1+\gamma}$, then $G$ is $(\eps/2)$-close to a complete multipartite graph $H$. By Lemma \ref{lem:pdisc change}, we have $$\disc^+(H)\leq \disc^+(G)+\frac{\eps}{2} n^2\leq n^{1+\gamma}+\frac{\eps}{2} n^2\leq \eps n^2$$ whenever $n$ is sufficiently large. Then, by Lemma~\ref{lem:from complete multipartite to Turan}, $H$ is $(\delta/2)$-close to a Tur\'an graph. It follows by the triangle inequality that $G$ is $\delta$-close to a Tur\'an graph.
\end{proof}

\section{Positive discrepancy of dense graphs} \label{sec:positive discrepancy}

In this section, we give a proof of Theorem~\ref{thm:positive disc}, which we restate here for convenience. 

\begin{custhm}{\ref{thm:positive disc}}
    Let $\eps > 0$. If $G$ is an $n$-vertex graph with average degree $d$ satisfying $1\leq d\leq (1/2 - \eps)n$, then $$\disc^+(G)\geq \Omega_\eps \left(nd^{1/3}\right).$$ 
    Moreover, if $G$ is $d$-regular, then the bisection width of $G$ is at most 
    $$\frac{dn}{4}-\Omega_{\eps}(nd^{1/3}).$$
\end{custhm}

The statement about the bisection width in Theorem \ref{thm:positive disc} follows from the asserted lower bound on $\disc^+(G)$, using \cite[Lemma 2.6]{raty2026positive}. It remains to prove this lower bound.

First, we prove a slightly extended eigenvalue version of this theorem. Our proof of the next result draws on ideas from~\cite{ihringer2023approximately} and~\cite{raty2026positive}, which establish the corresponding statement for regular graphs.

\begin{lemma} \label{prop:near regular eigenvalue for pdisc}
    Let $\eps > 0$ and let $G$ be an $n$-vertex graph with average degree $d\geq 1$ and maximum degree $\Delta \leq (1/2 - \eps)n$. Let $A$ be the adjacency matrix of $G$. If $d$ is sufficiently large with respect to $\eps$, then for any $q\in [0, 1/2 - \eps]$, the maximum eigenvalue of $A - qJ$ is at least $\Omega_\eps \left(d^{1/3}\right)$. 
\end{lemma}

\begin{proof}
 Note that we may assume that both $n$ and $d$ are sufficiently large with respect to $\eps$. Let $\mu_1\geq \dots\geq \mu_n$ denote the eigenvalues of $A - qJ$. We may assume that $\mu_1\leq d^{1/3}$, otherwise we are done. 

Let $\mathbf{e} = (1, \dots, 1)\in\mathbb{R}^n$, and let $P:= I - n^{-1}J$. Then $P$ is the matrix of orthogonal projection onto the subspace $V:= \langle \mathbf{e}\rangle^\perp\leq \mathbb{R}^n$. For a symmetric real $n\times n$ matrix $M$, let $M_P = PMP$ denote its projection onto $V$. 

Note that $\mathbf{e}$ is trivially an eigenvector of $(A - qJ)_P = A_P$ with eigenvalue zero. Let $\omega_1\geq \dots \geq \omega_{n - 1}$ denote the non-trivial eigenvalues of $A_P$ corresponding to eigenvectors in $V$. By Cauchy's interlacing theorem, we have $\mu_1\geq \omega_1\geq \mu_2\geq \dots \geq \omega_{n - 1}\geq \mu_n$.

\begin{claim}
$\mu_1\geq 2\eps$ and $\omega_{n-1}<0$.
\end{claim}

\begin{proof}
 As $G$ is nonempty, the matrix
    $$\begin{pmatrix}
        -q & 1-q \\
        1-q & -q
    \end{pmatrix}$$ is a submatrix of $A-qJ$. The largest eigenvalue of this matrix is $1-2q\geq 2\eps$, so by the Cauchy interlacing theorem, we have $\mu_1\geq 2\eps$.

In order to show that $\omega_{n-1}<0$, it is enough to prove that the trace of $A_P$ is negative. But this is true as $\trace(A_P)=\langle A,P\rangle=-d$.
\end{proof}

\begin{claim} \label{lem:mu_1 lower bound part 1}
    We have $\mu_1\geq \Omega_\eps\left(\sqrt{|\omega_{n - 1}|}\right)$. 
\end{claim}

\begin{proof}
    Let $X:= \mu_1 I + qJ - A$. Then $X\succeq 0$, so by the Schur product theorem, we have $X\circ X\succeq 0$. Note that 
    \begin{align}
        X\circ X &= (1 - 2q)A + q^2J + (\mu_1^2 + 2\mu_1 q)I. \label{eqn:Xsquare}
    \end{align}
    Let $u\in V$ be a unit eigenvector corresponding to the eigenvalue $\omega_{n - 1}$ of $(A - qJ)_P = A_P$. Using that $Pu=u$ and thus $u^TA_Pu=u^TAu$, equation (\ref{eqn:Xsquare}) implies $$u^T(X\circ X)u=(1-2q)u^TA_Pu+(\mu_1^2+2\mu_1 q)u^T u= (1-2q)\omega_{n-1}+\mu_1^2+2\mu_1 q.$$
    On the other hand, since $X\circ X$ is positive semidefinite, we have $u^T(X\circ X)u\geq 0$. Hence,
    \begin{equation}
        \mu_1^2+2\mu_1 q \geq (1 - 2q)|\omega_{n - 1}| \geq \Omega_\eps (|\omega_{n - 1}|). \label{eqn:mu and omega}
    \end{equation}
    Using that $\mu_1\geq 2\eps$ and $q\leq 1$, the left-hand-side is $O_{\eps}(\mu_1^2)$, finishing the proof.
\end{proof}

\begin{claim} \label{lem:mu_1 lower bound part 2}
    We have $\omega_1|\omega_{n - 1}|\geq \Omega_\eps(d)$. 
\end{claim}

\begin{proof}
    For any $\omega\in [\omega_{n - 1}, \omega_1]$, we have $(\omega - \omega_{n - 1})(\omega_1 - \omega)\geq 0$. That is, 
    \begin{equation} \label{rst lem 2-2 ineq}
        \omega^2\leq (\omega_1 + \omega_{n - 1})\omega + \omega_1|\omega_{n - 1}|. 
    \end{equation}
    Summing~(\ref{rst lem 2-2 ineq}) over $\omega = \omega_i$ for $i = 1, 2, \dots, n - 1$, we have
    \begin{equation*}
        \sum_{i = 1}^{n - 1}\omega_i^2\leq (\omega_1 + \omega_{n - 1})\left(\sum_{i = 1}^{n - 1}\omega_i\right) + (n - 1)\omega_1|\omega_{n - 1}|. 
    \end{equation*}
    In other words, 
    \begin{equation*}
        \trace(A_P^2) \leq (\omega_1 + \omega_{n - 1})\trace(A_P) + (n - 1)\omega_1|\omega_{n - 1}|. 
    \end{equation*}
    However, we have $\trace(A_P)=-d$  and 
    \begin{equation*}
        \trace(A_P^2) = \trace(APAP) = \trace(A^2) - 2n^{-1}\trace(A^2J) + n^{-2}\trace(AJAJ)\geq nd - 2\Delta d + d^2. 
    \end{equation*}
    Therefore, 
    \begin{equation*}
        nd - 2\Delta d + d^2\leq -d(\omega_1 + \omega_{n - 1}) + (n - 1)\omega_1|\omega_{n - 1}|. 
    \end{equation*}
    By Claim~\ref{lem:mu_1 lower bound part 1}, $|\omega_{n - 1}|\leq O_{\eps}(\mu_1^2)\leq O_\eps(d^{2/3})$. Therefore, when $d$ is sufficiently large with respect to $\eps$, we have $\omega_1 + \omega_{n-1}\geq -d$. Hence
    \begin{equation*}
        (n - 1)\omega_1|\omega_{n - 1}|\geq nd - 2\Delta d + d^2 + d(\omega_1 + \omega_{n - 1})\geq nd - 2\Delta d\geq 2\eps nd, 
    \end{equation*}
    which implies the desired inequality. 
\end{proof}
  By Claims~\ref{lem:mu_1 lower bound part 1} and~\ref{lem:mu_1 lower bound part 2}, we have $\mu_1^3\geq \omega_1 \mu_1^2\geq \Omega_\eps(\omega_1 |\omega_{n - 1}|)\geq \Omega_\eps(d)$, showing the required inequality $\mu_1\geq \Omega_{\eps}(d^{1/3})$. 
\end{proof}

Next, we prove Theorem~\ref{thm:positive disc} in the special case where the maximum degree $\Delta$ of $G$ satisfies $\Delta\leq (1 + \eps/10)d$.  Our proof follows an argument of R\"aty, Sudakov and Tomon~\cite{raty2026positive}. However, we use our main technical result, Theorem~\ref{thm:matrix bridge}, to deduce the desired lower bound on the positive discrepancy from the corresponding statement on the relevant eigenvalue. 

\begin{lemma} \label{prop:near regular positive disc}
    Let $\eps > 0$. If $G$ is an $n$-vertex graph with average degree $d$ and maximum degree $\Delta$, where $d\leq (1/2 - \eps)n$ and $\Delta \leq (1 + \eps/10)d$, then $\disc^+(G)\geq \Omega_\eps \left(nd^{1/3}\right)$. 
\end{lemma}
\begin{proof}
    If $d\leq n^{3/4}$, Lemma~\ref{prop:near regular positive disc} follows from Theorem~\ref{thm:three regimes}. Therefore, we may assume that $d\geq n^{3/4}$. In particular, we may assume that $d$ is sufficiently large with respect to $\eps$. 
    
    Let $p$ denote the density of $G$. By Lemma~\ref{lem:pdisc from vector}, it suffices to show that $\disc^+(A - p(J - I))\geq \Omega_\eps \left(nd^{1/3}\right)$. 
    Applying Theorem~\ref{thm:matrix bridge} to $M = A - p(J - I)$ with $n' = (1 - \eps/100)n$, it suffices to show that for any subset $S\subseteq [n]$ of size at least $n'$, the maximum eigenvalue of $M_S$ is at least $\Omega_\eps(d^{1/3})$. Let $H = G[S]$ have average degree $d'$ and maximum degree $\Delta'$. Since $e(H)\geq e(G) - (n - n')\Delta$, we have $d'\geq (1 - \eps/10)d$. In particular, $d'$ is sufficiently large with respect to $\eps$. We also have $$\Delta'\leq \Delta\leq \left(1 + \frac{\eps}{10}\right)d\leq \left(1 + \frac{\eps}{10}\right)\left(\frac{1}{2} - \eps \right)n\leq \left(\frac{1}{2} - \frac{\eps}{2}\right)v(H).$$ Therefore, applying Lemma~\ref{prop:near regular eigenvalue for pdisc} to $H$ with $\eps/2$ in place of $\eps$ shows that the maximum eigenvalue of $(A - pJ)_S$ is at least $\Omega_\eps (d'^{1/3})$. Therefore, the maximum eigenvalue of $M_S$ is at least $\Omega_\eps (d'^{1/3}) + p\geq \Omega_\eps(d^{1/3})$, as required. 
\end{proof}

Finally, we are ready to prove the main theorem of this section.

\begin{proof}[Proof of Theorem~\ref{thm:positive disc}]
    Without loss of generality, we assume that $\eps < 0.01$. The desired estimate follows from Theorem~\ref{thm:three regimes} when $d\leq n^{3/4}$. Therefore, we may assume that $d\geq n^{3/4}$, in particular, that $d$ is sufficiently large with respect to $\eps$. Moreover, we assume that $\disc^+(G)\leq nd^{1/3}$, otherwise, we are done. 
    
    Let $T$ be the set of vertices in $G$ with degree at least $(1 + \eps/40)d$, and let $H$ be the subgraph of $G$ induced on $V(G)\setminus T$. By Lemma \ref{lem:if weight on high degree}, there exist $c_1,c_2>0$ (depending only on $\eps$) such that
    \begin{equation*}
        \disc^+(G) \geq c_1(e(T) + e(T, T^c))-c_2n = c_1(e(G) - e(H))-c_2n\geq c_1d|T|/2-c_2n. 
    \end{equation*}
     This implies $|T|\leq O_{\eps}(nd^{-2/3})$ and
     $$e(H)=e(G)-e(T)-e(T,T^c)\geq e(G)-O_{\eps}(nd^{1/3}).$$
    Let $d'$ and $\Delta'$ be the average and maximum degree of $H$, respectively, and let $n' = v(H)$.
    By construction, $\Delta'\leq (1 + \eps/40)d$. To estimate $d'$, we note that 
    $$e(G) - O_\eps(nd^{1/3})\leq e(H)\leq e(G)\quad\text{and}\quad(1 - O_\eps(d^{-2/3}))n\leq n'\leq n.$$
    Therefore, when $d$ is sufficiently large, we have 
    $$d'=\frac{2e(H)}{n'}\leq \frac{2e(G)}{n(1-O_{\eps}(d^{-2/3}))}\leq d\left(1+O_{\eps}(d^{-2/3})\right).$$
    Similarly, $d'\geq (1 - O_\eps(d^{-2/3}))d$. Therefore, $\Delta'\leq (1 + \eps/20)d'$ when $d$ is sufficiently large. Applying Lemma~\ref{prop:near regular positive disc} with $\eps/2$ in place of $\eps$, we deduce that $\disc^+(H)\geq \Omega_\eps(n'
    d'^{1/3})\geq \Omega_\eps(nd^{1/3})$. Let $H'$ be the same graph as $H$ but with $|V(G)|-|V(H)|$ isolated vertices added. Then, by Lemma \ref{lem:pdisc change}, we have
    $$\disc^+(H)\leq \disc^+(H')\leq \disc^+(G)+(e(G)-e(H'))=\disc^+(G)+(e(G)-e(H)).$$
    Therefore, 
    \begin{equation*}
        \Omega_\eps(nd^{1/3})\leq \disc^+(H)\leq \disc^+(G) + (e(G) - e(H))\leq O_\eps(\disc^+(G) + n). 
    \end{equation*}
    Hence $\disc^+(G)\geq \Omega_\eps(nd^{1/3})$ as needed. 
\end{proof}

\section{MaxCut of $K_t$-free graphs} \label{sec:ktfree}

In this section, we prove Theorem \ref{thm:Kt free surplus bound}. The proof of this theorem has two ingredients: our new ``bridge'' between smallest eigenvalue and surplus, and the analogue of Theorem \ref{thm:Kt free surplus bound} proved in \cite{jin2025small} for the smallest eigenvalue.
We start with the former.

\begin{proposition} \label{thm:Kt free bridge result}
    Let $1\geq \alpha > \beta$ be positive real numbers, let $t\geq 3$ be a positive integer. Assume that for every $m\geq m_0(t,\alpha)$, every $K_t$-free $n$-vertex graph $G$ with $m$ edges satisfies $|\lambda_n|\geq m^\alpha/n$. Then for every $m\geq m_0(t,\alpha,\beta)$, every $K_t$-free graph $G$ with $m$ edges satisfies $\surp (G)\geq m^{\beta}$. 
\end{proposition}

\begin{proof}
    Let $G$ be an $n$-vertex $K_t$-free graph with $m$ edges. We may assume without loss of generality that $G$ has no isolated vertices. In particular, the average degree $d$ of $G$ satisfies $d\geq 1$. Let $\eps = \alpha - \beta > 0$. By Lemma~\ref{lem:regularization}, there exist constants $c, C$ such that either $\surp(G)\geq cn^{\alpha - \eps/2} d^{\alpha}$ or $G$ has an induced subgraph $\tilde{G}$ with $\tilde{n}$ vertices and average degree $\tilde{d}$, where $\tilde{n}^{\alpha - \eps/2} \tilde{d}^{\alpha} \geq cn^{\alpha - \eps/2}d^{\alpha}$ and $\Delta(\tilde{G}) \leq C\tilde{d}$. 

    For sufficiently large $m$, we have $cn^{\alpha - \eps/2} d^{\alpha} > m^{\beta}$. Therefore, if $\surp(G) < m^{\beta}$, then $G$ must have an induced subgraph $\tilde{G}$ with $\tilde{n}$ vertices and average degree $\tilde{d}$, where $\tilde{n}^{\alpha - \eps/2} \tilde{d}^{\alpha} \geq cn^{\alpha - \eps/2}d^{\alpha}$ and $\tilde{\Delta}:=\Delta(\tilde{G}) \leq C\tilde{d}$. Note that $e(\tilde{G})\geq \Omega(m^{\Omega(1)})$ is sufficiently large whenever $m$ is. 
    
    Let $n' = \left(1 - \frac{1}{10C}\right)\tilde{n}$. Any induced subgraph $H$ of $\tilde{G}$ with at least $n'$ vertices has at least $$e(\tilde{G}) - \tilde{\Delta}(\tilde{n} - v(H)) \geq \frac{1}{2}\tilde{n}\tilde{d} - \frac{1}{10C}\tilde{\Delta}\tilde{n}\geq \frac{1}{2}e(\tilde{G})$$ edges. By our assumption, the absolute value of the least eigenvalue of $H$ is at least
    \begin{equation*}
        e(H)^\alpha/v(H)\geq \Omega(e(\tilde{G})^\alpha/\tilde{n}). 
    \end{equation*}
    Therefore, by Theorem~\ref{thm:eigenvalue to surplus}, $\surp(\tilde{G})\geq \Omega((\tilde{n} - \lceil n'\rceil)e(\tilde{G})^\alpha/\tilde{n})\geq \Omega(e(\tilde{G})^\alpha)$. Hence 
    \begin{equation*}
        \surp(G)\geq \surp(\tilde{G})\geq \Omega(e(\tilde{G})^\alpha)\geq \Omega(\tilde{n}^{\alpha -\eps/2}\tilde{d}^\alpha)\geq \Omega(n^{\alpha - \eps/2} d^\alpha)\geq \Omega(m^{\alpha - \eps/2}). 
    \end{equation*}
    Hence for sufficiently large $m$, we have $\surp(G)\geq m^{\alpha - \eps} = m^{\beta}$. 
\end{proof}

We now state the key input from \cite{jin2025small}.

\begin{theorem}[Jin--Milojevi\'c--Tomon--Zhang {\cite[Theorem~1.2]{jin2025small}}] \label{thm:Kt free eigenvalue bound}
    For every $\gamma\in (0, 1/10)$, the following holds whenever $d$ is sufficiently large with respect to $\gamma$. Let $G$ be a graph of average degree $d$. If $|\lambda_n|\leq d^\gamma$, then $G$ contains a clique of size at least $d^{1 - 4\gamma}$. 
\end{theorem}

\begin{corollary} \label{cor:Kt free eigenvalue bound}
    Let $\alpha < 0.55$, let $t$ be a positive integer and let $m$ be sufficiently large with respect to $t$ and $\alpha$. Let $G$ be a $K_t$-free $n$-vertex graph with $m$ edges. Then $|\lambda_n|\geq m^\alpha/n$.
\end{corollary}

\begin{proof}
    Without loss of generality, assume that $\alpha > 1/2$ and $m > 0$. Since $|\lambda_n|\geq 1$, we may assume that $m^\alpha\geq n$. In particular, $d=2m/n\geq 2m^{1-\alpha}$ is sufficiently large. 
    
    We apply Theorem~\ref{thm:Kt free eigenvalue bound} to $\gamma = 2\alpha - 1$. If $|\lambda_n| < m^\alpha/n\leq n^{\alpha - 1}d^\alpha\leq d^{2\alpha - 1}$, then by Theorem~\ref{thm:Kt free eigenvalue bound}, $G$ must contain a clique of size at least $d^{1 - 4\gamma} > t$ when $d$ is sufficiently large, a contradiction. 
\end{proof}

Combining Proposition~\ref{thm:Kt free bridge result} and Corollary~\ref{cor:Kt free eigenvalue bound}, Theorem \ref{thm:Kt free surplus bound} follows immediately. 

\section{Concluding remarks}

Our proof of Theorem~\ref{thm:matrix bridge} is algorithmic. 
Let $n'\leq n$ be positive integers, let $M$ be an $n\times n$ real symmetric matrix and let $\lambda>0$ be a real number. After executing polynomially many (real-arithmetic) operations in $n$, we find either a vector $v\in [-1, 1]^n$ such that
\begin{equation*}
    v^TMv\geq (n - n')\lambda, 
\end{equation*}
or a subset $S\subseteq [n]$ of at least $n'$ elements for which the maximum eigenvalue of $M_{S}$ is less than $\lambda$. Indeed, start the algorithm with the zero vector $v^{(0)}=0$. At any stage with a vector $v^{(i)}$ satisfying $(v^{(i)})^TMv^{(i)}\geq \lambda\|v^{(i)}\|^2$, we proceed as follows. If the set $S_i=\{j\in [n]: v^{(i)}_j\neq \pm 1\}$ has the property that $M_{S_i}$ has maximum eigenvalue less than $\lambda$, we certify this with polynomially many operations. Otherwise, we find a non-zero vector $w$ supported on $S_i$ such that $w^T M w\geq \lambda\|w\|_2^2$, and we can add a scalar multiple of $w$ to $v^{(i)}$ to obtain a more saturated vector $v^{(i+1)}$ satisfying $(v^{(i+1)})^TMv^{(i+1)}\geq \lambda\|v^{(i+1)}\|^2$.

In particular, in the case of the negative adjacency matrix of graphs, with polynomially many operations in~$n$, we can find either a cut of surplus at least $(n - n')\lambda/4$ or an induced subgraph on at least $n'$ vertices with least eigenvalue $\mu$ such that $|\mu|< \lambda$.

\medskip

\paragraph{Acknowledgements.} We thank Igor Balla, Aleksa Milojevi\'c and Shengtong Zhang for helpful comments.

\medskip

\paragraph{AI declaration.} The mathematical content of the paper is due to the authors. ChatGPT Pro was used to polish an initial draft.

\bibliographystyle{abbrv}
\bibliography{mybib}

\end{document}